\documentclass[10pt]{amsart}

\usepackage{graphicx}
\usepackage{mathptmx}

\usepackage{amscd}
\usepackage{amsthm}
\usepackage{amsxtra}
\usepackage{a4wide}
\usepackage{latexsym}
\usepackage{amssymb}
\usepackage{amsfonts}
\usepackage{amsmath}
\usepackage{amsrefs}
\usepackage{mathrsfs}
\usepackage{color}
\usepackage{upref}
\usepackage{txfonts}
\usepackage{dsfont}

\usepackage[bookmarksnumbered, colorlinks, plainpages]{hyperref}
\hypersetup{colorlinks=true,linkcolor=red, anchorcolor=green, citecolor=cyan, urlcolor=red, filecolor=magenta, pdftoolbar=true}

\allowdisplaybreaks

\theoremstyle{plain}
\newtheorem{thm}{Theorem}[section]

\newtheorem{cor}[thm]{Corollary}

\theoremstyle{definition}

\numberwithin{thm}{section}
\numberwithin{equation}{section}

\def\esup{\operatornamewithlimits{ess\,sup}}

\def\up{\uparrow}

\def\mp{{\mathfrak M}}

\def\I{(0,\infty)}

\newcommand\norm[1]{\left\lVert#1\right\rVert}

\begin{document}

\title[New equivalence theorems for monotone quasilinear operators]{New equivalence theorems for weighted inequalities involving the composition of monotone quasilinear operators with the Hardy and Copson operators and their applications}

\author[R.Ch. Mustafayev]{RZA MUSTAFAYEV}
\address{RZA MUSTAFAYEV, Department of Mathematics, Kamil \"{O}zda\u{g} Faculty of Science, Karamano\u{g}lu Mehmetbey University, 70200, Karaman, Turkey}
\email{rzamustafayev@gmail.com}

\author[M. Yilmaz]{MERVE YILMAZ}
\address{MERVE YILMAZ, Department of Mathematics, Kamil \"{O}zda\u{g} Faculty of Science, Karamano\u{g}lu Mehmetbey University, 70200, Karaman, Turkey}
\email{mervegorgulu@kmu.edu.tr}


\begin{abstract}
In this paper, new equivalence theorems for the boundedness of the composition of a quasilinear operator $T$ with the Hardy and Copson operators in weighted Lebesgue spaces are proved. The usefulness of the obtained results is illustrated in the case of weighted Hardy-type and weighted iterated Hardy-type inequalities. 
\end{abstract}

\maketitle

\section{Introduction}\label{in}

There are plenty of motivations for studying weighted inequalities for a composition of quasilinear operators with the Hardy and Copson operators. Problems involving iterations of operators have recently been at the centre of the theory of weighted inequalities. Some of them are discussed in the last section of the paper. Other applications can be found in the theory of function spaces and interpolation theory.

By ${\mathfrak M}(0,\infty)$ we denote the set of all measurable functions on $(0,\infty)$. The symbol ${\mathfrak M}^+ (0,\infty)$ stands for the collection of all $f\in{\mathfrak M} (0,\infty)$ which are
non-negative on $(0,\infty)$, while ${\mathfrak M}^{+,\downarrow}(0,\infty)$ and
${\mathfrak M}^{+,\uparrow}(0,\infty)$ are used to denote the subset of ${\mathfrak M}^+ (0,\infty)$ containing all non-increasing and non-decreasing functions on $(0,\infty)$,
respectively. The family of all weight functions (also called just weights) on $(0,\infty)$, that is, almost everywhere positive functions
on $(0,\infty)$, is given by ${\mathcal W}(0,\infty)$.

For $p\in (0,\infty]$ and $w\in {\mathfrak M}^+ (0,\infty)$ we define the functional
$\|\cdot\|_{p,w,(0,\infty)}$ on ${\mathfrak M}(0,\infty)$ by
\begin{equation*}
\|f\|_{p,w,(0,\infty)} : = \left\{\begin{array}{cl}
\left(\int_{(0,\infty)} |f(x)|^p w(x)\,dx \right)^{\frac{1}{p}} & \qquad\mbox{if}\qquad p < \infty \\
\esup_{(0,\infty)} |f(x)|w(x) & \qquad\mbox{if}\qquad p = \infty.
\end{array}
\right.
\end{equation*}

If, in addition, $w\in {\mathcal W} (0,\infty)$, then the weighted Lebesgue space
$L^p(w,(0,\infty))$ is given by
\begin{equation*}
L^p(w,(0,\infty)) = \big\{f\in {\mathfrak M}(0,\infty):\,\, \|f\|_{p,w,(0,\infty)} < \infty \big\},
\end{equation*}
and it is equipped with the quasi-norm $\|\cdot\|_{p,w,(0,\infty)}$.

When $w\equiv 1$ on $(0,\infty)$, we write simply $L^p (0,\infty)$ and
$\|\cdot\|_{p,(0,\infty)}$ instead of $L^p(w,(0,\infty))$ and $\|\cdot\|_{p,w,(0,\infty)}$,
respectively.

In this paper we consider monotone quasilinear operators on ${\mathfrak M}^+ (0,\infty)$, that is, operators $T:{\mathfrak M}^+ (0,\infty) \rightarrow
{\mathfrak M}^+ (0,\infty)$ satisfying the following conditions:

{\rm (i)} $T(\lambda f) = \lambda Tf$ for all $\lambda \ge 0$ and $f
\in {\mathfrak M}^+(0,\infty)$;

{\rm (ii)} $Tf(x) \le C Tg(x)$ for almost all $x \in {\mathbb R}_+$
if $f(x) \le g(x)$ for almost all $x \in {\mathbb R}_+$, with
constant $C > 0$ independent of $f$ and $g$;

{\rm (iii)} $T(f+g) \le C (T f + Tg)$ for all $f,\,g \in {\mathfrak M}^+ (0,\infty)$, with a constant $C > 0$ independent of $f$ and $g$.

The aim of the paper is to present new equivalence theorems for the boundedness of the composition of a quasilinear opeator $T$ with the Hardy and Copson operators in weighted Lebesgue spaces, which allow to change inequalities 
\begin{equation}\label{eq.1}
\Bigg\| T \Bigg( \int_0^{x} h  \Bigg) \Bigg\|_{q,w,(0,\infty)} \le C \, \norm{h}_{p,v,(0,\infty)} \, , \quad h \in \mathfrak{M}^{+} (0,\infty) 
\end{equation}
and   
\begin{equation}\label{eq.2}
\Bigg\| T \Bigg( \int_{x}^{\infty}h \Bigg) \Bigg\|_{q,w,(0,\infty)} \le C \, \norm{h}_{p,v,(0,\infty)} \, , \quad h \in \mathfrak{M}^{+} (0,\infty)
\end{equation}
with an equivalent inequalities of the same type with weighted $L^1$-norms on the right-hand sides. 

The following two theorems allow to reduce the iterated inequalities
to the inequalities on the cone of monotone functions. 
\begin{thm}\cite[Theorem 3.1]{GogMusIHI}\label{RT.thm.main.3}
	Let $0 < q \le \infty$, $1 < p < \infty$, and $T$ be a monotone quasilinear operator on ${\mathfrak M}^+ (0,\infty)$. Assume that $w \in {\mathcal W}(0,\infty)$
	and $v \in {\mathcal W}(0,\infty)$ is such that
	\begin{equation}\label{RT.thm.main.3.eq.0}
	\int_0^x v^{1-p^{\prime}}(t)\,dt < \infty, \quad x \in(0,\infty).
	\end{equation}
	Then inequality \eqref{eq.1} holds iff
	\begin{equation}\label{RT.thm.main.3.eq.2}
	\Big\| T \Big( \Phi^2 f \Big) \Big\|_{q, w, (0,\infty)} \le C\, \| f \|_{p, \phi,(0,\infty)}, \quad f \in
	{\mathfrak M}^{+,\downarrow} (0,\infty)
	\end{equation}
	holds, where
	$$
	\phi (x) \equiv \phi\big[v;p\big](x) : = \Bigg( \int_0^x v^{1-{p}^{\prime}}(t)\,dt \Bigg)^{- \frac{p^{\prime}}{p^{\prime} + 1}}	v^{1-{p}^{\prime}}(x)
	$$
	and
	$$
	\Phi(x) \equiv \Phi \big[v;p\big](x) : = \int_0^x \phi(t)\,dt = \Bigg( \int_0^x	v^{1-{p}^{\prime}}(t)\,dt \Bigg)^{\frac{1}{p^{\prime} + 1}}
	$$
	for all $x \in(0,\infty)$.
\end{thm}

\begin{thm}\cite[Theorem 3.4]{GogMusIHI}\label{RT.thm.main.4}
	Let $0 < q \le \infty$, $1 < p < \infty$, and $T$ be a monotone quasilinear operator on ${\mathfrak M}^+ (0,\infty)$. Assume that $w \in {\mathcal W}(0,\infty)$
	and $v \in {\mathcal W}(0,\infty)$ is such that
	\begin{equation}\label{RT.thm.main.4.eq.0}
	\int_x^{\infty} v^{1-p^{\prime}}(t)\,dt < \infty, \quad x \in(0,\infty).
	\end{equation}
	Then inequality \eqref{eq.2} holds iff
	\begin{equation}\label{RT.thm.main.4.eq.2}
	\Big\| T \Big( \Psi^2 f \Big) \Big\|_{q, w, (0,\infty)} \le C\, \| f \|_{p, \psi,(0,\infty)}, \quad f \in {\mathfrak M}^{+,\uparrow}(0,\infty)
	\end{equation}
	holds, where
	$$
	\psi (x) \equiv \psi \big[v;p\big](x): = \Bigg( \int_x^{\infty}
	v^{1-{p}^{\prime}}(t)\,dt \Bigg)^{- \frac{p^{\prime}}{p^{\prime} + 1}}
	v^{1-{p}^{\prime}}(x)
	$$
	and
	$$
	\Psi(x) \equiv \Psi \big[v;p\big](x) : = \int_x^{\infty} \psi (t)\,dt = \Bigg( \int_x^{\infty}
	v^{1-{p}^{\prime}}(t)\,dt \Bigg)^{\frac{1}{p^{\prime} + 1}}
	$$
	for all $x \in(0,\infty)$.	
\end{thm}

Various equivalence statements were obtained using the previous two theorems in \cite{GogMusIHI}. We quote here the following ones, which are special cases of \cite[Corollary 3.3]{GogMusIHI} and \cite[Corollary 3.6]{GogMusIHI} when $\delta = p$, respectively.  
\begin{cor}\cite[Corollary 3.3]{GogMusIHI}\label{cor.1}
	Let $0 < q \le \infty$, $1 < p < \infty$, and $T$ be a monotone quasilinear operator on ${\mathfrak M}^+ (0,\infty)$. Assume that $w \in {\mathcal W}(0,\infty)$
	and $v \in {\mathcal W}(0,\infty)$ is such that \eqref{RT.thm.main.3.eq.0} holds. Then inequality \eqref{eq.1}  holds iff
	$$
	\Bigg\| T \Bigg( \Phi(x)^{2 \big(1 - \frac{1}{p}\big)} \bigg( \int_0^x h \bigg)^{\frac{1}{p}} \Bigg) \Bigg\|_{q,w,(0,\infty)} \le C \, \|h\|_{1,\Phi^{-1},(0,\infty)}, \quad h \in {\mathfrak M}^+ (0,\infty) 
	$$
	holds.
\end{cor} 

\begin{cor}\cite[Corollary 3.6]{GogMusIHI}\label{cor.2}
	Let $0 < q \le \infty$, $1 < p < \infty$, and $T$ be a monotone quasilinear operator on ${\mathfrak M}^+ (0,\infty)$. Assume that $w \in {\mathcal W}(0,\infty)$
	and $v \in {\mathcal W}(0,\infty)$ is such that \eqref{RT.thm.main.4.eq.0} holds. Then inequality \eqref{eq.2} holds iff
	$$
	\Bigg\| T \Bigg( \Psi(x)^{2 \big(1 - \frac{1}{p}\big)} \bigg( \int_x^{\infty} h \bigg)^{\frac{1}{p}} \Bigg) \Bigg\|_{q,w,(0,\infty)} \le C \, \|h\|_{1,\Psi^{-1},(0,\infty)}, \quad h \in {\mathfrak M}^+ (0,\infty) 
	$$
	holds.
\end{cor} 

Despite the fact that these results allow to change inequalities \eqref{eq.1} and \eqref{eq.2} with an equivalent inequalities of the same type with weighted $L^1$-norms on the right hand sides, however, an application of them when solving, for instance, weighted iterated Hardy type inequalities, provide us, in general, with so-called "flipped" conditions due to multipliers $\Phi^{2 (1 - {1} / {p})}$ and $\Psi^{2 (1 - {1} / {p})}$ on the left hand sides of them, respectively. 

In this paper, we make a simple yet surprisingly efficient observation which enables us to get statements avoiding these "unwanted" expressions on the left-hand sides of equivalent inequalities. 

The main results of the paper reads as follows:

\begin{thm}\label{main.thm.1}
	Let $ 0 < q < \infty $, $1 < p < \infty$, and $T$ be a monotone quasilinear operator on ${\mathfrak M}^+ (0,\infty)$. Assume that $w \in \mathcal{W}(0,\infty)$ and $v \in \mathcal{W}(0,\infty)$ is such that \eqref{RT.thm.main.3.eq.0} holds.  
	Then inequality \eqref{eq.1} holds if and only if 
	$$
	\Bigg\| \Bigg( T\bigg( \int_0^{x} h  \bigg)^{\frac{1}{p}} \Bigg)^p \Bigg\|_{q/p,w,(0,\infty)} \le C^p \, \norm{h}_{1, \Phi^{1 - 2p},\I} \, , \quad h \in \mathfrak{M}^{+} (0,\infty) 
	$$
	holds.
\end{thm}

\begin{thm}\label{main.thm.2}
	Let $ 0 < q < \infty $, $1 < p < \infty$, and $T$ be a monotone quasilinear operator on ${\mathfrak M}^+ (0,\infty)$. Assume that $w \in \mathcal{W}(0,\infty)$ and $v \in \mathcal{W}(0,\infty)$ is such that \eqref{RT.thm.main.4.eq.0} holds.
	Then inequality \eqref{eq.2} holds if and only if 
	$$
	\Bigg\| \Bigg( T\bigg( \int_{x}^{\infty}h \bigg)^{\frac{1}{p}} \Bigg)^p \Bigg\|_{q/p,w,(0,\infty)} \le C^p \, \norm{h}_{1, \Psi^{1 - 2p},\I} \, , \quad h \in \mathfrak{M}^{+} (0,\infty)
	$$
	holds.
\end{thm} 

Note that similar results can be proved when $q = \infty$.

\begin{thm}\label{main.thm.1.0.0.0}
	Let $1 < p < \infty$, and $T$ be a monotone quasilinear operator on ${\mathfrak M}^+ (0,\infty)$. Assume that $w \in \mathcal{W}(0,\infty)$ and $v \in \mathcal{W}(0,\infty)$ is such that \eqref{RT.thm.main.3.eq.0} holds.
    Then the inequality
	\begin{equation*}
	\Bigg\| T \Bigg( \int_0^{x} h  \Bigg) \Bigg\|_{\infty,w,(0,\infty)} \le C \, \norm{h}_{p,v,(0,\infty)} \, , \quad h \in \mathfrak{M}^{+} (0,\infty) 
	\end{equation*}
	holds if and only if 
	$$
	\Bigg\| \Bigg( T\bigg( \int_0^{x} h  \bigg)^{\frac{1}{p}} \Bigg)^p \Bigg\|_{\infty,w^p,(0,\infty)} \le C^p \, \norm{h}_{1, \Phi^{1 - 2p},\I} \, , \quad h \in \mathfrak{M}^{+} (0,\infty) 
	$$
	holds.
\end{thm}

\begin{thm}\label{main.thm.2.0.0.0}
	Let $1 < p < \infty$, and $T$ be a monotone quasilinear operator on ${\mathfrak M}^+ (0,\infty)$. Assume that $w \in \mathcal{W}(0,\infty)$ and $v \in \mathcal{W}(0,\infty)$ is such that \eqref{RT.thm.main.4.eq.0} holds.
	Then the inequality 
	\begin{equation*}
	\Bigg\| T \Bigg( \int_{x}^{\infty}h \Bigg) \Bigg\|_{\infty,w,(0,\infty)} \le C \, \norm{h}_{p,v,(0,\infty)} \, , \quad h \in \mathfrak{M}^{+} (0,\infty)
	\end{equation*}
	holds if and only if 
	$$
	\Bigg\| \Bigg( T\bigg( \int_{x}^{\infty}h \bigg)^{\frac{1}{p}} \Bigg)^p \Bigg\|_{\infty,w^p,(0,\infty)} \le C^p \, \norm{h}_{1, \Psi^{1 - 2p},\I} \, , \quad h \in \mathfrak{M}^{+} (0,\infty)
	$$
	holds.
\end{thm}

Our approach to prove these new equivalence theorems is somehow an extension of the ideas from \cite{GogMusIHI}. To achieve the goal, we expand certain known reduction theorems from \cite{GogStep} in the next section of the present paper.

Obtained results show that in order to solve \eqref{eq.1} and \eqref{eq.2}, it is enough to know solutions of "similar" iterated inequalities with weighted $L^1$-norms on the right-hand sides only. This surprising property of the above-mentioned inequalities is demonstrated in the cases of weighted Hardy-type and weighted iterated Hardy-type inequalities in the last section of the paper.

Throughout the paper, we always denote by  $C$ a positive
constant, which is independent of main parameters but it may vary
from line to line. By $a\lesssim b$, ($b\gtrsim a$) we mean that $a\leq \lambda b$, where
$\lambda >0$ depends on inessential parameters. If $a\lesssim b$ and
$b\lesssim a$, we write $a\approx b$ and say that $a$ and $b$ are
equivalent.  We will denote by $\mathds{1}$ the function ${\mathds{1}}(x) =
1$, $x \in (0,\infty)$. Unless a special remark is made, the
differential element $dx$ is omitted when the integrals under
consideration are the Lebesgue integrals. We put $0 \cdot \infty = 0$, $\infty / \infty =
0$ and $0/0 = 0$. If $p\in (0,+\infty)$, we define $p' = p / (1 - p)$ when $p \in (0,1)$, $p' = p / (p - 1)$ when $p \in (1, \infty)$ and $p' = \infty$ when $p = 1$.

The paper is organised as follows. In Section \ref{RT}, we prove new reduction theorems.  Proofs of the main results of the paper are presented in Section \ref{MR}. In Section \ref{appl}, we apply new equivalency statements to the characterizations of weighted Hardy-type inequalities, give some historical remarks on the weighted iterated Hardy-type inequalities and illustrate an application of our main theorems to the solutions of these inequalities.


\section{Note on reduction theorems}\label{RT}

Many works have lately been devoted to the study of weighted inequalities for classical operators restricted on the cones of monotone functions, which are important in the study of Lorentz spaces; see, for example, \cite{AM}, \cite{CRS}, \cite{sawyer19990} and the recent survey \cite{GogStep}, as well as the literature listed there. The reduction approach has been a crucial instrument in studying such inequalities almost since the beginning, according to which a given inequality on cones of monotone functions can be equivalently replaced with some inequality of non-negative functions that is easier to characterise than the original.
		
The Sawyer duality principle \cite{sawyer19990} is a universal reduction tool for positive linear operators, allowing one to reduce a $L^p - L^q$ inequality for monotone functions with $1 < p,\, q < \infty$ to a more manageable inequality for arbitrary non-negative functions. This principle was extended by Stepanov in \cite{step_1993} to the case $0 < p < 1 < q < \infty$. In the same work, Stepanov applied a different approach to this problem, known as reduction theorems, which allowed the range of parameters to be extended to $1 < p < \infty$, $0 < q < \infty$. The case $p \le q$, $0 < p \le 1$ was alternatively characterized in \cite{Step_1993_JLMS}, \cite{carsor_1993}, \cite{burgold}, \cite{lai_1993}. Later on, some direct reduction theorems were found in \cite{gogpick_2207}, \cite{cgmp_2008}, \cite{GogStep1} and \cite{GogStep} involving supremum operators which work for the case $0 < q < p \le 1$.		
		
The following two statements are generalizations of the case when $p=1$ of \cite[Theorem 3.2]{GogStep} and \cite[Theorem 3.4]{GogStep}, respectively, and can be proved in the same way (cf. \cite[Theorems 2.3 and 2.5]{GogMusIHI}). 

\begin{thm} \label{theorem1}
	Let $ 0 < q \le \infty $, $ 0 < \alpha < \infty$, $w,v \in \mathcal{W}(0,\infty)$, and $T$ be a monotone quasilinear operator on ${\mathfrak M}^+ (0,\infty)$. Then the inequality 
	$$
	\norm{Tf}_{q,w,(0,\infty)} \le C \, \norm{f}_{1,v,\I}, \quad f \in \mathfrak{M}^{+,\downarrow} (0,\infty)
	$$
	holds if and only if the inequality
	$$
	\Bigg\| T \Bigg(\frac{1}{V(x)^{\alpha + 1}} \int_0^{x} h \, V^{\alpha} \Bigg) \Bigg\|_{q,w,(0,\infty)} \le C \, \norm{h}_{1,\I}, \quad h \in \mathfrak{M}^{+} (0,\infty) 
	$$
	holds, where
	$$
	V(x) : = \int_0^x v(t)\,dt, \quad x \in (0,\infty).
	$$
\end{thm}

\begin{thm} \label{theorem2}
	Let $ 0 < q \le \infty $, $ 0 < \alpha < \infty$, $w,v \in \mathcal{W}(0,\infty)$, and $T$ be a monotone quasilinear operator on ${\mathfrak M}^+ (0,\infty)$. Then the inequality 
	$$
	\norm{Tf}_{q,w,(0,\infty)} \le C \, \norm{f}_{1,v,\I}, \quad f \in \mathfrak{M}^{+,\up} (0,\infty)
	$$
	holds if and only if the inequality
	$$
	\Bigg\| T \Bigg(\frac{1}{V_{*}(x)^{\alpha + 1}} \int_{x}^{\infty} h \, V_{*}^{\alpha} \Bigg) \Bigg\|_{q,w,(0,\infty)} \le C \, \norm{h}_{1,\I}, \quad h \in \mathfrak{M}^{+} (0,\infty) 
	$$
	holds, where
	$$
	V_*(x) : = \int_x^{\infty} v(t)\,dt, \quad x \in (0,\infty).
	$$
\end{thm}

Note that these statements were proved in \cite{GogStep} for $\alpha = 1$. Since \cite[Theorem 3.4]{GogStep} was presented without proof, for the convenience of the reader we give here the complete proof of Theorem \ref{theorem2}. 

\noindent {\bf Proof of Theorem \ref{theorem2}:} We begin with the sufficiency part. Assume that the inequality
		$$
		\Bigg\| T \Bigg(\frac{1}{V_{*}(x)^{\alpha + 1}} \int_{x}^{\infty} h \, V_{*}^{\alpha} \Bigg) \Bigg\|_{q,w,(0,\infty)} \le C \, \norm{h}_{1,\I} 
		$$
		holds for all $h \in \mathfrak{M}^{+} (0,\infty)$. 
		
		Since for any $x \in (0,\infty)$
		$$
		1 = \frac{\alpha +1}{V_{*}(x)^{\alpha + 1}} \int_{x}^{\infty} V_{*}^{\alpha} \, v,
		$$
		then the inequality
		$$
		f(x) \le \frac{\alpha +1}{V_{*}(x)^{\alpha + 1}} \int_{x}^{\infty} f \, V_{*}^{\alpha} \, v
		$$
		holds for any $f \in \mathfrak{M}^{+,\up} (0,\infty)$. 
		
		By using (ii), we get that 
		\begin{align*}
		\norm{Tf}_{q,w,\I} & \lesssim \Bigg\| T \, \Bigg(\frac{1}{V_{*}(x)^{\alpha + 1}} \int_{x}^{\infty} f \, V_{*}^{\alpha} \, v \Bigg) \Bigg\|_{q,w,\I} \le C \norm{f \, v}_{1,\I} = C \norm{f}_{1,v,\I}.
		\end{align*}
		
		To prove the necessity, assume that the inequality 
		$$
		\norm{Tf}_{q,w,(0,\infty)} \le C \, \norm{f}_{1,v,\I}
		$$
		holds for all $f \in \mathfrak{M}^{+,\up} (0,\infty)$. 
		
		First suppose that $V_{*}(0) = \infty$. 
		
		Since for any $x \in (0,\infty)$
		$$
		\frac{1}{V_{*}(x)^{\alpha + 1}} = (\alpha + 1) \int_{0}^{x} \frac{v}{V_{*}^{\alpha + 2}},
		$$
		then 
		\begin{align*}
		\frac{1}{V_{*}(x)^{\alpha + 1}} \int_{x}^{\infty} h \, V_{*}^{\alpha}  & = (\alpha + 1) \int_{0}^{x} \frac{v}{V_{*}^{\alpha + 2}} \int_{x}^{\infty} h \, V_{*}^{\alpha} \\
		& \le (\alpha + 1) \int_{0}^{x} \frac{v(t)}{V_{*}(t)^{\alpha + 2}} \bigg(\int_{t}^{\infty} h \, V_{*}^{\alpha}\bigg) \, dt.
		\end{align*}
		
		Set the function
		$$
		f(x) : = (\alpha + 1) \int_{0}^{x} \frac{v(t)}{V_{*}(t)^{\alpha + 2}} \bigg(\int_{t}^{\infty} h \, V_{*}^{\alpha}\bigg) \, dt, \quad x \in \I.
		$$
		Clearly $f \in \mathfrak{M}^{+,\up} (0,\infty)$. 
		
		By condition (ii) and Fubini theorem, we get that
		\begin{align*}
			\Bigg\| T \Bigg(\frac{1}{V_{*}(x)^{\alpha + 1}} \int_{x}^{\infty} h \, V_{*}^{\alpha} \Bigg) \Bigg\|_{q,w,\I}  & \lesssim \norm{Tf}_{q,w,\I} \\ 
			& \le C \, \norm{f}_{1,v,\I} \\
			& = (\alpha + 1) \, C \, \int_{0}^{\infty} \Bigg(\int_{0}^{x} \frac{v(t)}{V_{*}(t)^{\alpha + 2}} \Bigg( \int_{t}^{\infty} h \, V_{*}^{\alpha} \Bigg) \, dt \Bigg) v(x) \, dx \\
			& = (\alpha + 1) \, C \, \int_{0}^{\infty} \frac{v(t)}{V_{*}(t)^{\alpha + 2}} \Bigg(\int_{t}^{\infty} h \, V_{*}^{\alpha} \Bigg) V_{*}(t)\, dt \\
			& = (\alpha + 1) \, C \, \int_{0}^{\infty} \frac{v(t)}{V_{*}(t)^{\alpha + 1}} \Bigg(\int_{t}^{\infty} h \, V_{*}^{\alpha} \Bigg) \, dt \\
			& = (\alpha + 1) \, C \, \int_{0}^{\infty} h(x) \, V_{*}(x)^{\alpha} \Bigg( \int_{0}^{x} \frac{v}{V_{*}^{\alpha + 1}} \Bigg) \, dx \\
			& = \frac{(\alpha + 1)}{\alpha} \, C \, \int_{0}^{\infty} h(x) \, V_{*}(x)^{\alpha} \frac{1}{V_{*}(x)^{\alpha}} \, dx \\
			& = \frac{(\alpha + 1)}{\alpha} \, C \, \norm{h}_{1,\I}.
		\end{align*}
	
		Next suppose that $V_{*}(0) \in \I $. So we can write for any $x \in (0,\infty)$
		\begin{align*}
			\frac{1}{V_{*}(x)^{\alpha + 1}} \int_{x}^{\infty} h \, V_{*}^{\alpha} & = \Bigg[\frac{1}{V_{*}(x)^{\alpha + 1}} - \frac{1}{V_{*}(0)^{\alpha + 1}} \Bigg] \int_{x}^{\infty} h \, V_{*}^{\alpha}	+\frac{1}{V_{*}(0)^{\alpha + 1}} \int_{x}^{\infty} h \, V_{*}^{\alpha} \\
			& = (\alpha + 1)  \, \int_{0}^{x} \frac{v}{V_{*}^{\alpha + 2}} \int_{x}^{\infty} h \, V_{*}^{\alpha} +\frac{1}{V_{*}(0)^{\alpha + 1}} \int_{x}^{\infty} h \, V_{*}^{\alpha} \\
			& \le (\alpha + 1) \, \int_{0}^{x} \frac{v(t)}{V_{*}(t)^{\alpha + 2}} \Bigg(\int_{t}^{\infty} h \, V_{*}^{\alpha} \Bigg) \, dt + \frac{1}{V_{*}(0)^{\alpha + 1}} \int_{0}^{\infty} h \, V_{*}^{\alpha}  \\
			& = f(x) +\lambda \mathds{1}(x),
		\end{align*}
		where
		$$
		\lambda := \frac{1}{V_{*}(0)^{\alpha + 1}} \int_{0}^{\infty} h \, V_{*}^{\alpha}.
		$$
		
		By using condition (iii), we have that 
		$$
		\Bigg\| T\Bigg(\frac{1}{V_{*}(x)^{\alpha + 1}} \int_{x}^{\infty} h \, V_{*}^{\alpha} \Bigg) \Bigg\|_{q,w,\I}  \lesssim C \, \bigg(\norm{Tf}_{q,w,\I} + \norm{T(\lambda \mathds{1})}_{q,w,\I}\bigg).
		$$
		
		Consider the estimates of the right-hand side of the inequality, term by term. 
		
		The preceding discussion gives the estimation
		$$
		\norm{Tf}_{q,w,\I} \lesssim \norm{h}_{1,\I}.
		$$
		
		To estimate the remaining term, we make use of condition (i), obtaining
		\begin{align*}
			\norm{T\big(\lambda \mathds{1}\big)}_{q,w,\I} & = \lambda \, \norm{T\mathds{1}}_{q,w,\I} \\
			& \le \lambda \, C \, \norm{\mathds{1}}_{1,v,\I} \\
			& = C \, \frac{1}{V_{*}(0)^{\alpha + 1}} \Bigg( \int_{0}^{\infty} h \, V_{*}^{\alpha} \Bigg) \, V_{*}(0) \\
			& = C \, \frac{1}{V_{*}(0)^{\alpha}} \int_{0}^{\infty} h \, V_{*}^{\alpha} \\
			& \le C \, \int_{0}^{\infty} h \, V_{*}^{\alpha}  \frac{1}{V_{*}^{\alpha}} \\
			& = C \, \norm{h}_{1,\I}.
		\end{align*}
		
		Combining the estimates yields the result
		$$
		\Bigg\| T \Bigg(\frac{1}{V_{*}(x)^{\alpha + 1}} \int_{x}^{\infty} h \, V_{*}^{\alpha} \Bigg) \Bigg\|_{q,w,(0,\infty)} \le C \, \norm{h}_{1,\I} ,
		$$
		and the proof is complete. \qed
	

\section{Main results}\label{MR}

In this section, we prove our main results. Since the proofs of Theorems \ref{main.thm.1} and \ref{main.thm.2} are similar, we prove the latter.

\noindent {\bf Proof of Theorem \ref{main.thm.2}:} By Theorem \ref{RT.thm.main.4}, the inequality 
$$
\Bigg\| T \Bigg( \int_{x}^{\infty}h \Bigg) \Bigg\|_{q,w,(0,\infty)} \le C \, \norm{h}_{p,v,(0,\infty)} \, , \quad h \in \mathfrak{M}^{+} (0,\infty)
$$
holds if and only if
$$
\Big\| T \Big(\Psi^2 f \Big) \Big\|_{q,w,(0,\infty)} \le C \, \norm{f}_{p,\psi,(0,\infty)}, \quad f \in \mathfrak{M}^{+,\up} (0,\infty)
$$
holds.

Clearly, the latter holds iff the inequality 
\begin{equation}\label{eq.1.1}
\Bigg\| \Bigg( T \bigg( \Psi^2 f^{\frac{1}{p}} \bigg) \Bigg)^p \Bigg\|_{q/p,w,(0,\infty)} \le C^p \, \norm{f}_{1,\psi,(0,\infty)}, \quad f \in \mathfrak{M}^{+,\up} (0,\infty)
\end{equation}
is satisfied.

It is easy to see that the operator $T_1$ defined by
$$
T_1 f = \Bigg( T \bigg( \Psi^2 f^{\frac{1}{p}} \bigg) \Bigg)^p
$$
is monotone quasilnear operator on $\mathfrak{M}^{+} (0,\infty)$.

With this new notation inequality \eqref{eq.1.1} takes the form
\begin{equation}\label{eq.1.2}
\|T_1 f\|_{q / p,w,(0,\infty)} \le C^p \, \|f\|_{1,\psi,(0,\infty)}, \quad f \in \mathfrak{M}^{+,\up} (0,\infty).
\end{equation}

Applying Theorem \ref{theorem2} with $T_1$ and $\alpha = 2p-1 $, we see that inequality \eqref{eq.1.2} holds if and only if the inequality
\begin{equation}\label{eq.1.3}
\Bigg\| T_1 \Bigg( \frac{1}{\Psi(x)^{2p}} \int_{x}^{\infty}h \, \Psi^{2p-1} \Bigg) \Bigg\|_{q / p,w,(0,\infty)} \le C^p \, \|h\|_{1,(0,\infty)}, \quad f \in \mathfrak{M}^+ (0,\infty)
\end{equation}
is satisfied.

Since
\begin{align*}
T_1 \Bigg( \frac{1}{\Psi(x)^{2p}} \int_{x}^{\infty}h \, \Psi^{2p-1} \Bigg) = \Bigg( T \Bigg( \Psi(x)^2 \bigg(\frac{1}{\Psi(x)^{2p}} \int_{x}^{\infty}h \, \Psi^{2p-1} \bigg)^{\frac{1}{p}} \Bigg) \Bigg)^p = \Bigg( T\bigg( \int_{x}^{\infty} h \, \Psi^{2p-1} \bigg)^{\frac{1}{p}} \Bigg)^p, 
\end{align*}
equation \eqref{eq.1.3} can be rewritten as
$$
\Bigg\| \Bigg( T\bigg( \int_{x}^{\infty} h \, \Psi^{2p-1} \bigg)^{\frac{1}{p}} \Bigg)^p \Bigg\|_{q/p,w,(0,\infty)} \le C^p \, \norm{h}_{1,\I} \, , \quad h \in \mathfrak{M}^{+} (0,\infty),
$$
which is clearly equivalent to the inequality
$$
\Bigg\| \Bigg( T\bigg( \int_{x}^{\infty} h \bigg)^{\frac{1}{p}} \Bigg)^p \Bigg\|_{q/p,w,(0,\infty)} \le C^p \, \norm{h}_{1,\Psi^{1 - 2p},\I} \, , \quad h \in \mathfrak{M}^{+} (0,\infty).
$$

The proof is completed. \qed
	

\section{Applications}\label{appl}

In this  section, we illustrate an applications of Theorems \ref{main.thm.1} and \ref{main.thm.2} to the solutions of the weighted Hardy-type and the weighted iterated Hardy-type inequalities.

The usefulness of our main results can be seen even when $T = I$, where $I$ is the identity operator on ${\mathfrak M}^+ (0,\infty)$. Clearly, $T$ satisfies conditions {\rm (i) - (iii)}. Applying Theorem \ref{main.thm.1} and Theorem \ref{main.thm.2}, respectively, to the operator $T$, we get the following two equivalency statements.

\begin{thm}\label{main.thm.1.0.0}
	Let $0 < q < \infty$ and $1 < p < \infty$. Assume that $w \in \mathcal{W}(0,\infty)$ and $v \in \mathcal{W}(0,\infty)$ is such that \eqref{RT.thm.main.3.eq.0} holds. 
	Then the inequality 
	\begin{equation}\label{eq.Hardy}
	\Bigg( \int_0^{\infty} \bigg( \int_0^x h \Bigg)^q w(x) \,dx \Bigg)^{\frac{1}{q}} \le C \, \Bigg( \int_0^{\infty} h^p v \Bigg)^{\frac{1}{p}}, \quad h \in {\mathfrak M}^+(0,\infty),
	\end{equation}
	holds if and only if the inequality
	\begin{equation}\label{main.0.1.0}
	\Bigg( \int_0^{\infty} \Bigg( \int_0^x h \Bigg)^{\frac{q}{p}} w(x)\,dx \Bigg)^{\frac{p}{q}} \leq C^p \, \int_0^{\infty} h \Phi^{1-2p}, \qquad h \in {\mathfrak M}^+(0,\infty)
	\end{equation}
	holds.
\end{thm}

\begin{thm}\label{main.thm.2.0.0}
	Let $0 < q < \infty$ and $ 1 < p < \infty $. Assume that $w \in \mathcal{W}(0,\infty)$ and $v \in \mathcal{W}(0,\infty)$ is such that \eqref{RT.thm.main.4.eq.0} holds.
	Then the inequality 
	\begin{equation}\label{eq.Copson}
	\bigg( \int_0^{\infty} \bigg( \int_x^{\infty} h \bigg)^q w(x)\,dx\bigg)^{\frac{1}{q}} \le C \, \bigg( \int_0^{\infty} h^p v \bigg)^{\frac{1}{p}}, \quad  h \in {\mathfrak M}^+(0,\infty),
	\end{equation}
	holds if and only if the inequality
	\begin{equation}\label{main.0.2.0}
	\Bigg( \int_0^{\infty} \Bigg( \int_t^{\infty} h \Bigg)^{\frac{q}{p}} w(x)\,dx \Bigg)^{\frac{p}{q}} \leq C^p \, \int_0^{\infty} h \Psi^{1-2p}, \qquad h \in {\mathfrak M}^+(0,\infty)
	\end{equation}
	holds.
\end{thm}

Recall the following characterizations of inequalities 
\begin{equation}
\bigg( \int_0^{\infty} \bigg( \int_x^{\infty} h \bigg)^q w(x)\,dx\bigg)^{{1} / {q}} \le C \, \int_0^{\infty} h v, \qquad h \in {\mathfrak M}^+(0,\infty)
\end{equation}
and 
\begin{equation}
\bigg( \int_0^{\infty} \bigg( \int_0^x h \bigg)^q w(x)\,dx\bigg)^{{1} / {q}} \le C \, \int_0^{\infty} h v, \qquad h \in {\mathfrak M}^+(0,\infty)
\end{equation}
(see, \cite[Theorem 1]{Bradley} for $1 \le q$ and \cite[Theorem 3.3]{ss} for $q < 1$).
\begin{thm}\label{thm.BradleySinnamon}
	Let $0 < q < \infty$ and $v,\,w \in {\mathcal W} (0,\infty)$.
	
	{\rm (i)} Let $1 \le q$. Then 
	$$
	\sup_{h \in \mp^+ (0,\infty)} \frac{\bigg( \int_0^{\infty} \bigg( \int_0^x h \bigg)^q w(x)\,dx\bigg)^{{1} / {q}}}{\int_0^{\infty} h v} = \sup_{x \in (0,\infty)} \Bigg( \int_x^{\infty} w \Bigg)^{{1} / {q}} \Bigg( \esup_{t \in (0,x]} v(t)^{-1} \Bigg).
	$$
	
	{\rm (ii)} Let $q < 1$. Then 
	$$
	\sup_{h \in \mp^+ (0,\infty)} \frac{\bigg( \int_0^{\infty} \bigg( \int_0^x h \bigg)^q w(x)\,dx\bigg)^{{1} / {q}}}{\int_0^{\infty} h v} \approx \Bigg( \int_0^{\infty} \Bigg( \int_x^{\infty} w \Bigg)^{q'} \, w(x) \, \Bigg( \esup_{t \in (0,x]} v(t)^{-1} \Bigg)^{q'} \,dx \Bigg)^{{1} / {q'}}.
	$$
\end{thm}	

\begin{thm}\label{thm.BradleySinnamon_dual}
	Let $0 < q < \infty$ and $v,\,w \in {\mathcal W} (0,\infty)$.
	
	{\rm (i)} Let $1 \le q$. Then 
	$$
	\sup_{h \in \mp^+ (0,\infty)} \frac{\bigg( \int_0^{\infty} \bigg( \int_x^{\infty} h \bigg)^q w(x)\,dx\bigg)^{{1} / {q}}}{\int_0^{\infty} h v} = \sup_{x \in (0,\infty)} \Bigg( \int_0^x w \Bigg)^{{1} / {q}} \Bigg( \esup_{t \in [x,\infty)} v(t)^{-1} \Bigg).
	$$
	
	{\rm (ii)} Let $q < 1$. Then 
	$$
	\sup_{h \in \mp^+ (0,\infty)} \frac{\bigg( \int_0^{\infty} \bigg( \int_x^{\infty} h \bigg)^q w(x)\,dx\bigg)^{{1} / {q}}}{\int_0^{\infty} h v} \approx \Bigg( \int_0^{\infty} \Bigg( \int_0^x w \Bigg)^{q'} \, w(x) \, \Bigg( \esup_{t \in [x,\infty)} v(t)^{-1} \Bigg)^{q'} \,dx \Bigg)^{{1} / {q'}}.
	$$
\end{thm}	

Combining Theorems \ref{main.thm.1.0.0} and \ref{main.thm.2.0.0} with Theorems \ref{thm.BradleySinnamon} and \ref{thm.BradleySinnamon_dual}, respectively, we get (now classical) well-known characterizations of weights for which weighted Hardy inequality \eqref{eq.Hardy} and  weighted Copson inequality \eqref{eq.Copson} hold (cf. \cite{ok,kp,kufmalpers,kufperssam}).
\begin{thm}\label{thm.Hardy}
	Let $1 < p < \infty$, $0 < q < \infty$ and $v,\,w \in {\mathcal W} (0,\infty)$. 
	
	{\rm (a)} Let $p \le q$. Then 
	$$
	\sup_{h \in \mp^+} \frac{\bigg( \int_0^{\infty} \bigg( \int_0^x h \bigg)^q w(x)\,dx\bigg)^{\frac{1}{q}}}{\bigg( \int_0^{\infty} h^p v \bigg)^{\frac{1}{p}}} \approx \sup_{x \in (0,\infty)} \Bigg( \int_x^{\infty} w \Bigg)^{\frac{1}{q}} \Bigg( \int_0^x v^{1-p'} \Bigg)^{\frac{1}{p}}.
	$$
	
	{\rm (b)} Let $q < p$. Then 
	$$
	\sup_{h \in \mp^+} \frac{\bigg( \int_0^{\infty} \bigg( \int_0^x h \bigg)^q w(x)\,dx\bigg)^{\frac{1}{q}}}{\bigg( \int_0^{\infty} h^p v \bigg)^{\frac{1}{p}}} \approx \Bigg( \int_0^{\infty} \Bigg( \int_x^{\infty} w \Bigg)^{\frac{q}{p - q}} \, w(x) \, \Bigg( \int_0^x v^{1-p'} \Bigg)^{\frac{q(p-1)}{p-q}} \,dx \Bigg)^{\frac{p-q}{pq}}.
	$$
\end{thm}	

\begin{thm}\label{thm.Copson}
	Let $1 < p < \infty$, $0 < q < \infty$ and $v,\,w \in {\mathcal W} (0,\infty)$. 
	
	{\rm (a)} Let $p \le q$. Then 
	$$
	\sup_{h \in \mp^+} \frac{\bigg( \int_0^{\infty} \bigg( \int_x^{\infty} h \bigg)^q w(x)\,dx\bigg)^{\frac{1}{q}}}{\bigg( \int_0^{\infty} h^p v \bigg)^{\frac{1}{p}}} \approx \sup_{x \in (0,\infty)} \Bigg( \int_0^x w \Bigg)^{\frac{1}{q}} \Bigg( \int_x^{\infty} v^{1-p'} \Bigg)^{\frac{1}{p}}.
	$$
	
	{\rm (b)} Let $q < p$. Then 
	$$
	\sup_{h \in \mp^+} \frac{\bigg( \int_0^{\infty} \bigg( \int_x^{\infty} h \bigg)^q w(x)\,dx\bigg)^{\frac{1}{q}}}{\bigg( \int_0^{\infty} h^p v \bigg)^{\frac{1}{p}}} \approx \Bigg( \int_0^{\infty} \Bigg( \int_0^x w \Bigg)^{\frac{q}{p - q}} \, w(x) \, \Bigg( \int_x^{\infty} v^{1-p'} \Bigg)^{\frac{q(p-1)}{p-q}} \,dx \Bigg)^{\frac{p-q}{pq}}.
	$$
\end{thm}	

In the remaining part of the paper, we sketch how to get solutions of weighted iterated Hardy-type inequalities from their counterparts with weighted $L^1$-norms on the right-hand sides. The investigation of such inequalities started with the study of the inequality
\begin{equation}\label{mainn0}
\Bigg( \int_0^{\infty} \Bigg(\int_0^x \Bigg( \int_t^{\infty} h \Bigg)^r u(t)\,dt \Bigg)^{\frac{q}{r}} w(x)\,dx\Bigg)^{\frac{1}{q}} \le C \Bigg(\int_0^{\infty} h^p v \Bigg)^{\frac{1}{p}}, \qquad h \in \mp^+(0,\infty). 
\end{equation}
Inequality \eqref{mainn0} have been considered in the case $r=1$ in \cite{gop2009} (see also \cite{g1}), where the result was
presented without proof, and in the case $p=1$ in \cite{gjop} and \cite{ss}, where the special
type of weight $v$ was considered. Recall that the inequality has been completely characterized	in \cite{GMP1} and \cite{GMP2} in the case $0 < r < \infty$, $0 < q \leq \infty$, $1 \le p < \infty$ by using discretization and anti-discretization methods; but, the obtained results were restricted to non-degenerate weights. Another approach to get the characterization of inequality \eqref{mainn0} was presented in \cite{ProkhStep1}. However, this characterization involves auxiliary functions, which make conditions more complicated. 

As it was mentioned in \cite{GogMusIHI} the characterization of "dual" inequality
\begin{equation}\label{main}
\Bigg( \int_0^{\infty} \Bigg( \int_x^{\infty} \Bigg( \int_0^t h \Bigg)^r u(t)\,dt
\Bigg)^{\frac{q}{r}} w(x)\,dx \Bigg)^{\frac{1}{q}}\leq C \, \Bigg( \int_0^{\infty} h^p v \Bigg)^{\frac{1}{p}}, \qquad h \in {\mathfrak M}^+(0,\infty)
\end{equation}
can be easily obtained from the solutions of inequality \eqref{mainn0}, which was presented in \cite{GKPS}. 

Another pair of "dual" weighted iterated Hardy-type inequalities are
\begin{equation}\label{iterH1}
\Bigg( \int_0^{\infty} \Bigg(\int_x^{\infty} \Bigg( \int_t^{\infty} h \Bigg)^r u(t)\,dt \Bigg)^{\frac{q}{r}} w(x)\,dx \Bigg)^{\frac{1}{q}} \le C \Bigg( \int_0^{\infty} h^p v \Bigg)^{\frac{1}{p}}, \qquad h \in \mp^+ (0,\infty)
\end{equation}
and 
\begin{equation}\label{iterH3}
\Bigg( \int_0^{\infty} \Bigg( \int_0^x \Bigg( \int_0^t h \Bigg)^r u(t)\,dt \Bigg)^{\frac{q}{r}} w(x)\,dx \Bigg)^{\frac{1}{q}} \le C \Bigg( \int_0^{\infty} h^p v \Bigg)^{\frac{1}{p}}, \qquad h \in \mp^+(0,\infty).
\end{equation}

The weighted Hardy-type inequalities on the cones of non-increasing or non-decreasing functions can be used to characterize all four inequalities, when $1 < p < \infty$. This approach provides solution of iterated inequalities by so-called "flipped" conditions (see, \cite{GogMusIHI} and \cite{gog.mus.2017_2}). In the case when $p = 1$, \cite{GogMusIHI} contains solutions of inequalities \eqref{mainn0} - \eqref{iterH3} with weight functions $\int_0^x v$ and $\big( \int_x^{\infty} v\big)^{-1}$ on the right-hand side, as well.

Different approach to solve \eqref{iterH1} has been given in \cite{mus.2017_cor} when $p=1$ using a combination of reduction techniques and discretization. The "classical" conditions ensuring the validity of \eqref{iterH1} was recently presented in \cite{krepick}. Inequalities \eqref{main} and \eqref{iterH3} were recently characterized by using discretization techniques in \cite{GMPTU} and \cite{GU}, respectively. The characterization of inequality \eqref{main} by using a combination of reduction techniques and discretization was lately presented in \cite{mus.yil.2022}.

It is worth noting that the characterizations of inequalities \eqref{mainn0} - \eqref{iterH3} are important because many inequalities
for classical operators  can be reduced to them. These inequalities play an important role in the theory of weighted Morrey-type
spaces and Ces\`{a}ro function spaces (see \cite{GKPS}, \cite{gmu_CMJ}, \cite{gmu_2017} and \cite{gogmusunv}). Note that using characterizations of weighted Hardy inequalities it is easy to obtain the characterization of the boundedness of bilinear Hardy-type inequalities (see, for instance, \cite{AOR}, \cite{Krep} and \cite{BMU}).

Consider the operator $T : \mathfrak{M}^+ (0,\infty) \rightarrow \mathfrak{M}^+ (0,\infty)$, defined for some $0 < r < \infty$, by
$$
T f (x) = \Bigg( \int_x^{\infty} f(t)^r u(t)\,dt \Bigg)^{\frac{1}{r}}, \quad f \in \mathfrak{M}^+ (0,\infty), \quad x \in (0,\infty).
$$	
Obvıously, $T$ satisfies conditions {\rm (i) - (iii)}. 

Applying Theorem \ref{main.thm.1} and Theorem \ref{main.thm.2}, respectively, to the operator $T$, we get the following two equivalency statements.

\begin{thm}\label{main.thm.1.0}
	Let $0 < q,\, r < \infty$ and $1 < p < \infty$. Assume that $u,\,w \in \mathcal{W}(0,\infty)$ and $v \in \mathcal{W}(0,\infty)$ is such that \eqref{RT.thm.main.3.eq.0} holds. 
	Then inequality \eqref{main} holds if and only if the inequality
	\begin{equation}\label{main.0.1}
	\Bigg( \int_0^{\infty} \Bigg( \int_x^{\infty} \Bigg( \int_0^t h \Bigg)^{\frac{r}{p}} u(t)\,dt
	\Bigg)^{\frac{q}{r}} w(x)\,dx \Bigg)^{\frac{p}{q}} \leq C^p \, \int_0^{\infty} h \Phi^{1-2p}, \qquad h \in {\mathfrak M}^+(0,\infty)
	\end{equation}
	holds.
\end{thm}

\begin{thm}\label{main.thm.2.0}
	Let $0 < q,\, r < \infty$ and $ 1 < p < \infty $. Assume that $u,\,w \in \mathcal{W}(0,\infty)$ and $v \in \mathcal{W}(0,\infty)$ is such that \eqref{RT.thm.main.4.eq.0} holds.
	Then inequality \eqref{iterH1} holds if and only if the inequality
	\begin{equation}\label{main.0.2}
	\Bigg( \int_0^{\infty} \Bigg( \int_x^{\infty} \Bigg( \int_t^{\infty} h \Bigg)^{\frac{r}{p}} u(t)\,dt
	\Bigg)^{\frac{q}{r}} w(x)\,dx \Bigg)^{\frac{p}{q}} \leq C^p \, \int_0^{\infty} h \Psi^{1-2p}, \qquad h \in {\mathfrak M}^+(0,\infty)
	\end{equation}
	holds.
\end{thm}

Recall the following statement on characterization of the inequality	
\begin{equation}\label{main.1}
\Bigg( \int_0^{\infty} \Bigg( \int_x^{\infty} \Bigg( \int_0^t h \Bigg)^r u(t)\,dt
\Bigg)^{\frac{q}{r}} w(x)\,dx \Bigg)^{\frac{1}{q}}\leq C \, \int_0^{\infty} h v, \qquad h \in {\mathfrak M}^+(0,\infty)
\end{equation}
(see, for instance, \cite{GU} and \cite{mus.yil.2022}).
	
\begin{thm}\label{mainthm1}
	Let $0 < q,\, r < \infty$ and  $u,\,v,\,w \in {\mathcal W}\I$.
	
	{\rm (a)} Let $1 \le \min\{q,\,r\}$. Then inequality \eqref{main.1} holds if and only if $F_1^1 < \infty$ and $F_2^1 < \infty$, where
	\begin{equation*}
		F_1^1 := \sup_{x \in (0,\infty)} \Bigg( \int_0^x u \Bigg)^{\frac{1}{q}} \Bigg( \int_x^{\infty} w \Bigg)^{\frac{1}{r}} \Bigg( \esup_{y \in (0,x]} v(y)^{-1} \Bigg),
	\end{equation*}
	and
	\begin{equation*}
		F_2^1 := \sup_{x \in (0,\infty)} \Bigg( \int_x^{\infty} \Bigg( \int_t^{\infty} w \Bigg)^{\frac{q}{r}} u(t)\,dt \Bigg)^{\frac{1}{q}} \Bigg( \esup_{y \in (0,x]} v(y)^{-1} \Bigg).
	\end{equation*}
	
	Moreover, if $C$ is the best constant in \eqref{main.1}, then $C \approx F_1^1 + F_2^1$.
	
	{\rm (b)} Let $q < 1 \le r$. Then inequality \eqref{main.1} holds if and only if $F_3^1 < \infty$ and $F_4^1 < \infty$, where
	\begin{equation*}
		F_3^1 := \Bigg( \int_0^{\infty} \Bigg( \sup_{t \in [x,\infty)} \Bigg( \int_t^{\infty} w \Bigg)^{\frac{1}{r}} \Bigg( \esup_{y \in (0,t]} v(y)^{-1} \Bigg) \Bigg)^{q'} \Bigg( \int_0^x u \Bigg)^{q'} u(x) \,dx \Bigg)^{\frac{1}{q'}},
	\end{equation*}
	and
	\begin{equation*}
		F_4^1 := \Bigg(\int_0^{\infty} \Bigg(\int_{x}^{\infty} \Bigg( \int_y^{\infty} w \Bigg)^{\frac{q}{r}} u(y) \, dy \Bigg)^{q'} \Bigg( \int_x^{\infty} w	\Bigg)^{\frac{q}{r}} \Bigg( \esup_{y \in (0,x]} v(y)^{-1} \Bigg)^{q'} \, u(x)\,dx \Bigg)^{\frac{1}{q'}}.
	\end{equation*}
	
	Moreover, if $C$ is the best constant in \eqref{main.1}, then $C \approx F_3^1 + F_4^1$.
	
	{\rm (c)} Let $r < 1 \le q$. Then inequality \eqref{main.1} holds if and only if $F_2 < \infty$ and $F_5 < \infty$, where
	\begin{equation*}
		F_5^1 : = \sup_{t \in (0,\infty)} \Bigg( \int_0^t u \Bigg)^{\frac{1}{q}} \Bigg( \int_t^{\infty} \Bigg( \int_x^{\infty} w \Bigg)^{r'} w(x) \Bigg( \esup_{y \in (0,x]} v(y)^{-1} \Bigg)^{r'} \,dx \Bigg)^{\frac{1}{r'}}.
	\end{equation*}
	
	Moreover, if $C$ is the best constant in \eqref{main.1}, then $C \approx F_2^1 + F_5^1$.	
	
	{\rm (d)} Let $\max\{q,\,r\} < 1$. Then inequality \eqref{main.1} holds if and only if $F_4 < \infty$ and $F_6 < \infty$, where
	\begin{align*}
		F_6^1 & := \Bigg( \int_0^{\infty} \Bigg( \int_0^t u \Bigg)^{q'} u(t) \Bigg( \int_t^{\infty} \Bigg( \int_x^{\infty} w \Bigg)^{r'} w(x) \Bigg( \esup_{y \in (0,x]} v(y)^{-1} \Bigg)^{r'} \,dx \Bigg)^{\frac{q'}{r'}}\,dt \Bigg)^{\frac{1}{q'}}.	
	\end{align*}
	
	Moreover, if $C$ is the best constant in \eqref{main.1}, then $C \approx F_4^1 + F_6^1$.	
\end{thm}

Theorem \ref{main.thm.1.0} allows to obtain solution of inequality \eqref{main} from solution of inequality \eqref{main.1}.

\begin{thm}
	Let $0 < q,\, r < \infty$ and $1 < p < \infty$. Assume that $u,\,w \in \mathcal{W}(0,\infty)$ and $v \in \mathcal{W}(0,\infty)$ is such that \eqref{RT.thm.main.3.eq.0} holds. 
	
	{\rm (a)} Let $p \le \min\{q,\,r\}$. Then inequality \eqref{main} holds if and only if $F_1 < \infty$ and $F_2 < \infty$, where
	\begin{equation*}
	F_1 := \sup_{x \in (0,\infty)} \Bigg( \int_0^x w \Bigg)^{\frac{1}{q}} \Bigg( \int_x^{\infty} u \Bigg)^{\frac{1}{r}} \Bigg( \int_{0}^x v^{1-p'} \Bigg)^{\frac{1}{p'}},
	\end{equation*}
	and
	\begin{equation*}
	F_2 := \sup_{x \in (0,\infty)} \Bigg( \int_x^{\infty} \Bigg( \int_t^{\infty} u \Bigg)^{\frac{q}{r}} w(t) \, dt \Bigg)^{\frac{1}{q}} \Bigg( \int_{0}^x v^{1-p'} \Bigg)^{\frac{1}{p'}}.
	\end{equation*}
	Moreover, if $C$ is the best constant in \eqref{main}, then $C \approx F_1 + F_2$.
	
	{\rm (b)} Let $q < p \le r$. Then inequality \eqref{main} holds if and only if $F_3 < \infty$ and $F_4 < \infty$, where
	\begin{equation*}
	F_3 := \Bigg( \int_0^{\infty} \Bigg( \sup_{t \in [x,\infty)} \Bigg( \int_t^{\infty} u \Bigg)^{\frac{1}{r}} \Bigg( \int_{0}^t v^{1-p'} \Bigg)^{\frac{1}{p'}} \Bigg)^{\frac{pq}{p-q}} \Bigg( \int_0^x w \Bigg)^{\frac{q}{p-q}} w(x) \, dx \Bigg)^{\frac{p-q}{pq}},
	\end{equation*}
	and
	\begin{equation*}
	F_4 := \Bigg(\int_0^{\infty} \Bigg(\int_{x}^{\infty} \Bigg( \int_t^{\infty} u \Bigg)^{\frac{q}{r}} w(t) \, dt \Bigg)^{\frac{q}{p-q}} \Bigg( \int_x^{\infty} u \Bigg)^{\frac{q}{r}} \Bigg( \int_{0}^x v^{1-p'} \Bigg)^{\frac{(p-1)q}{p-q}} \, w(x) \, dx \Bigg)^{\frac{p-q}{pq}}.
	\end{equation*}
	
	Moreover, if $C$ is the best constant in \eqref{main}, then $C \approx F_3 + F_4$.
	
	{\rm (c)} Let $r < p \le q$. Then inequality \eqref{main} holds if and only if $F_2 < \infty$ and $F_5 < \infty$, where
	\begin{equation*}
	F_5 : = \sup_{t \in (0,\infty)} \Bigg( \int_0^t w \Bigg)^{\frac{1}{q}} \Bigg( \int_t^{\infty} \Bigg( \int_x^{\infty} u \Bigg)^{\frac{r}{p-r}} u(x) \Bigg( \int_{0}^x v^{1-p'} \Bigg)^{\frac{(p-1)r}{p-r}} \, dx \Bigg)^{\frac{p-r}{pr}}.
	\end{equation*}
	
	Moreover, if $C$ is the best constant in \eqref{main}, then $C \approx F_2 + F_5$.	
	
	{\rm (d)} Let $\max\{q,\, r\} < p$. Then inequality \eqref{main} holds if and only if $F_4 < \infty$ and $F_6 < \infty$, where
	\begin{align*}
	F_6 & := \Bigg( \int_0^{\infty} \Bigg( \int_0^t w \Bigg)^{\frac{q}{p-q}} w(t) \Bigg( \int_t^{\infty} \bigg( \int_x^{\infty} u \Bigg)^{\frac{r}{p-r}} u(x) \Bigg( \int_{0}^x v^{1-p'} \Bigg)^{\frac{(p-1)r}{p-r}} \, dx \Bigg)^{\frac{q(p-r)}{r(p-q)}} \, dt \Bigg)^{\frac{p-q}{pq}}.	
	\end{align*}
	
	Moreover, if $C$ is the best constant in \eqref{main}, then $C \approx F_4 + F_6$.	
	
\end{thm}

\begin{proof}
	The proof of the statement immediately follows from Theorems \ref{main.thm.1.0} and \ref{mainthm1}.
\end{proof}

We need the following statement on characterization of the inequality
\begin{equation}\label{iterH1.1}
\Bigg( \int_0^{\infty} \Bigg(\int_x^{\infty} \Bigg( \int_t^{\infty} h \Bigg)^r u(t)\,dt \Bigg)^{\frac{q}{r}} w(x)\,dx \Bigg)^{\frac{1}{q}} \le C \, \int_0^{\infty} h v, \qquad h \in \mp^+ (0,\infty)
\end{equation}
(see, for example, \cite{krepick} and \cite{mus.2017_cor}).

\begin{thm}\label{mainthm2}
	Let $0 < q,\, r < \infty$ and  $u,\,v,\,w \in {\mathcal W}\I$.
	
	{\rm (a)} Let $1 \le \min\{q,\,r\}$. Then inequality \eqref{iterH1.1} holds if and only if $E_1^1 < \infty$, where
	\begin{align*}
		E_1^1 : & = \esup_{t\in (0,\infty)} \Bigg(\int_{0}^{t} u(s) \, \Bigg(\int_{s}^{t} w \Bigg)^{\frac{q}{r}} \, ds \Bigg)^{\frac{1}{q}} \Bigg( \esup_{y\in [t,\infty)} v(y)^{-1} \Bigg).
	\end{align*}	
	
	Moreover, if $C$ is the best constant in \eqref{iterH1.1}, then $C \approx E_1^1$.
	
	{\rm (b)} Let $q < 1 \le r$. Then inequality \eqref{iterH1.1} holds if and only if $E_2^1 < \infty$ and $E_3^1 < \infty$, where
	\begin{equation*}
		E_2^1 := \Bigg(\int_{0}^{\infty} \Bigg(\int_{0}^{t} u \Bigg)^{q'} u(t) \Bigg( \esup_{x\in [t,\infty)} \Bigg(\int_{t}^{x} w \Bigg)^{\frac{q'}{r}} \Bigg( \esup_{y\in [x,\infty)} v(y)^{-1} \Bigg)^{q'} \Bigg) \, dt \Bigg)^{\frac{1}{q'}},
	\end{equation*}
	and
	\begin{equation*}
		E_3^1 := \Bigg(\int_{0}^{\infty} \Bigg(\int_{0}^{t} u(s) \Bigg(\int_{s}^{t} w \Bigg)^{\frac{q}{r}} \, ds \Bigg)^{q'} u(t) \Bigg( \esup_{x\in [t,\infty)} \Bigg(\int_{t}^{x} w \Bigg)^{\frac{q}{r}} \Bigg( \esup_{y\in [x,\infty)} v(y)^{-1} \Bigg)^{q'} \Bigg) \, dt \Bigg)^{\frac{1}{q'}}.	
	\end{equation*}
	
	Moreover, if $C$ is the best constant in \eqref{iterH1.1}, then $C \approx E_2^1 + E_3^1$.
	
	{\rm (c)} Let $r < 1 \le q$. Then inequality \eqref{iterH1.1} holds if and only if $E_1^1 < \infty$ and $E_4^1 < \infty$, where
	\begin{align*}
		E_4^1 & := \esup_{t\in (0,\infty)} \Bigg(\int_{0}^{t} u \Bigg)^{\frac{1}{q}} \Bigg(\int_{t}^{\infty} \Bigg(\int_{t}^{x} w \Bigg)^{r'} w(x) \Bigg(\esup_{y\in [x,\infty)} v(y)^{-1} \Bigg)^{r'} \, dx \Bigg)^{\frac{1}{r'}}.	
	\end{align*}
	
	Moreover, if $C$ is the best constant in \eqref{iterH1.1}, then $C \approx E_1^1 + E_4^1$.	
	
	{\rm (d)} Let $\max\{q,\,r\} < 1$. Then inequality \eqref{iterH1.1} holds if and only if $E_3^1 < \infty$ and $E_5^1 < \infty$, where
	\begin{align*}
		E_5^1 & := \Bigg(\int_{0}^{\infty} \Bigg(\int_{0}^{t} u \Bigg)^{q'} u(t) \Bigg( \int_{t}^{\infty} \Bigg(\int_{t}^{x} w \Bigg)^{r'} w(x) \Bigg( \esup_{y\in [x,\infty)} v(y)^{-1} \Bigg)^{r'} \, dx \Bigg)^{\frac{q'}{r'}} \, dt \Bigg)^{\frac{1}{q'}}.	
	\end{align*}
	
	Moreover, if $C$ is the best constant in \eqref{iterH1.1}, then $C \approx E_3^1 + E_5^1$.	
\end{thm}

Similarly, Theorem \ref{main.thm.2.0} allows to obtain solution of inequality \eqref{iterH1} from solution of inequality \eqref{iterH1.1}.

\begin{thm}
	Let $0 < q,\, r < \infty$ and $1 < p < \infty$. Assume that $u,\,w \in \mathcal{W}(0,\infty)$ and $v \in \mathcal{W}(0,\infty)$ is such that \eqref{RT.thm.main.3.eq.0} holds.
	
	{\rm (a)} Let $p \le \min\{q,\,r\}$. Then inequality \eqref{iterH1} holds if and only if $E_1 < \infty$, where
	\begin{equation*}
		E_1 :  = \esup_{t\in (0,\infty)} \Bigg( \int_{0}^{t} w(s) \Bigg( \int_{s}^{t} u \Bigg)^{\frac{q}{r}} \, ds \Bigg)^{\frac{1}{q}}  \Bigg( \int_t^{\infty} v^{1-p'} \Bigg)^{\frac{1}{p'}}.
	\end{equation*}	
	
	Moreover, if $C$ is the best constant in \eqref{iterH1}, then $C \approx E_1$.
	
	{\rm (b)} Let $q < p \le r$. Then inequality \eqref{iterH1} holds if and only if $E_2 < \infty$ and $E_3 < \infty$, where
	\begin{equation*}
		E_2  := \Bigg( \int_{0}^{\infty} \Bigg( \int_{0}^{t} w \Bigg)^{\frac{q}{p-q}} w(t) \Bigg( \esup_{x\in (t,\infty)} \Bigg( \int_{t}^{x} u \Bigg)^{\frac{1}{r}}  \Bigg( \int_x^{\infty} v^{1-p'} \Bigg)^{\frac{1}{p'}} \Bigg)^{\frac{pq}{p-q}} \, dt \Bigg)^{\frac{p-q}{pq}} 
	\end{equation*}
	and 
	\begin{equation*}
		E_3  := \Bigg( \int_{0}^{\infty} \Bigg( \int_{0}^{t} w(s) \Bigg(\int_{s}^{t} u \Bigg)^{\frac{q}{r}} \, ds \Bigg)^{\frac{q}{p-q}} w(t) \Bigg( \esup_{x\in (t,\infty)} \Bigg( \int_{t}^{x} u \Bigg)^{\frac{q}{r}} \Bigg( \int_x^{\infty} v^{1-p'} \Bigg)^{\frac{(p-1)q}{p-q}} \Bigg) \, dt \Bigg)^{\frac{p-q}{pq}}.	
	\end{equation*}
	
	Moreover, if $C$ is the best constant in \eqref{iterH1}, then $C \approx E_2 + E_3$.
	
	{\rm (c)} Let $r < p \le q$. Then inequality \eqref{iterH1} holds if and only if $E_1 < \infty$ and $E_4 < \infty$, where
	\begin{align*}
		E_4 & := \esup_{t\in (0,\infty)} \Bigg( \int_{0}^{t} w \Bigg)^{\frac{1}{q}} \Bigg( \int_{t}^{\infty} \Bigg(\int_{t}^{x} u \Bigg)^{\frac{r}{p-r}} u(x) \Bigg( \int_x^{\infty} v^{1-p'} \Bigg)^{\frac{(p-1)r}{p-r}}  \, dx \Bigg)^{\frac{p-r}{pr}}.	
	\end{align*}
	
	Moreover, if $C$ is the best constant in \eqref{iterH1}, then $C \approx E_1 + E_4$.	
	
	{\rm (d)} Let $\max\{q,\,r\} < p$. Then inequality \eqref{iterH1} holds if and only if $E_3 < \infty$ and $E_5 < \infty$, where
	\begin{align*}
		E_5 & := \Bigg( \int_{0}^{\infty} \Bigg( \int_{0}^{t} w \Bigg)^{\frac{q}{p-q}} w(t) \Bigg( \int_{t}^{\infty} \Bigg( \int_{t}^{x} u \Bigg)^{\frac{r}{p-r}} u(x) \Bigg( \int_x^{\infty} v^{1-p'} \Bigg)^{\frac{(p-1)r}{p-r}} \, dx \Bigg)^{\frac{q(p-r)}{r(p-q)}} \, dt \Bigg)^{\frac{p-q}{pq}}.	
	\end{align*}
	
	Moreover, if $C$ is the best constant in \eqref{iterH1}, then $C \approx E_3 + E_5$.	
\end{thm}

\begin{proof}
	The proof of the statement at once follows from Theorems \ref{main.thm.2.0} and \ref{mainthm2}.
\end{proof}


\end{document}